\documentclass[final,12pt]{elsarticle}
\usepackage{xcolor,soul,cancel}
\usepackage[color]{showkeys}
\definecolor{refkey}{rgb}{0,1,1}
\definecolor{labelkey}{rgb}{1,0,0}
\usepackage{amssymb}
\usepackage{amsmath}
\usepackage{amsthm}
\usepackage{graphicx}
\usepackage{float}
\journal{arXiv}

\newtheorem{thm}{Theorem}

\newtheorem{cor}{Corollary}
\newtheorem{prop}[thm]{Proposition}
\newtheorem{Ex}{Example}
\numberwithin{equation}{section}
\newcommand{\eq} [1] {\begin{equation}\label{#1}\quad}
\newcommand{\en} {\end{equation}}
\newcommand{\scal}[1]{\langle#1\rangle}
\newcommand{\norm}[1]{\left\Vert#1\right\Vert}
\newcommand{\abs}[1]{\left\vert#1\right\vert}

\newcommand{\C}{\mathbb C}
\newcommand{\M}{{\bf M}}
\newcommand{\R}{\mathbb R}

\newcommand{\im}{\operatorname{Im}}

\newcommand{\re}{\operatorname{Re}}

\begin{document}

\begin{frontmatter}
%\ead{ims2@nyu.edu, ilya@math.wm.edu, imspitkovsky@gmail.com}

\title{Kippenhahn curves of some tridiagonal matrices \tnoteref{support}}
%----------Author 1

%\ead{ims2@nyu.edu, ilya@math.wm.edu, imspitkovsky@gmail.com}
\author[MUC]{Nat\'alia Bebiano}
\address[MUC]{Bepartamento de Matem\'atica, Universidade da Coimbra, Portugal}
\ead{bebiano@mat.uc.pt}
\author[FUC]{Jo\'ao da Provid\'encia}
\address[FUC]{Departamento de F\'isica, Universidade da Coimbra, Portugal}
\ead{providencia@uc.pt}
\author[nyuad]{Ilya M. Spitkovsky}
\ead{ims2@nyu.edu, ilya@math.wm.edu, imspitkovsky@gmail.com}
\address[nyuad]{Division of Science and Mathematics, New York  University Abu Dhabi (NYUAD), Saadiyat Island,
P.O. Box 129188 Abu Dhabi, United Arab Emirates}
\author[nyuad]{Kenya~Vazquez}
\ead{kvn222@nyu.edu}
\tnotetext[support]{The results are partially based on the Capstone project of [KV] under the supervision of [IMS]. The latter was also supported in part by Faculty Research funding from the Division of Science and Mathematics, New York University Abu Dhabi. The work of the first author [NB] was supported by the Centre for Mathematics of the University of Coimbra ---
UIDB/00324/2020, funded by the Portuguese Government through FCT/MCTES.}
\begin{abstract}
Tridiagonal matrices with constant main diagonal and reciprocal pairs of off-diagonal entries are considered. Conditions for such matrices with sizes up to 6-by-6 to have elliptical numerical ranges are obtained.
\end{abstract}

\end{frontmatter}

\section{Introduction} 

Let $\M_n$ stand for the algebra of all $n$-by-$n$ matrices $A$ with the entries $a_{ij}\in\C$, $i,j=1,\ldots n$. We will identify $A\in\M_n$ with a linear operator acting on $\C^n$, the latter being equipped with the standard scalar product $\scal{\cdot,\cdot}$ and the associated norm $\norm{x}:=\scal{x,x}^{1/2}$. The {\em numerical range} of $A$ is defined as \eq{nr} W(A)=\{\scal{Ax,x}\colon \norm{x}=1\}, \en see e.g. \cite[Chapter 1]{HJ2} or more recent \cite[Chapter~6]{DGSV}  for the basic properties of $W(A)$, in particular its convexity (the Toeplitz-Hausdorff theorem) and invariance under unitary similarities. 

For our purposes it is important that $W(A)$ is the convex hull of a certain curve $C(A)$ associated with the matrix $A$, described as follows. Using the standard notation \[ \re A=\frac{A+A^*}{2},\quad  \im A=\frac{A-A^*}{2i},\] let $\lambda_1(\theta),\ldots,\lambda_n(\theta)$ be the spectrum of $\re(e^{i\theta}A)$, counting the multiplicities. Then the tangent lines of $C(A)$ with the slope 
$\cot\theta$ are $\{e^{-i\theta}(\lambda_k(\theta)+it)\colon t\in\R\}$, $k=1,\ldots,n$. 
We will thus call the characteristic polynomial of $\re(e^{i\theta}A)$, \eq{nrgp} P_A(\lambda,\theta)=\det\left(\re(e^{i\theta}A)-\lambda I\right), \en  the {\em NR generating polynomial} of $A$. The connection between $C(A)$ and $W(A)$ was established by Kippenhahn \cite{Ki} (see also the English translation \cite{Ki08}) and following \cite[Chapter~13]{DGSV} we will call $C(A)$ the {\em Kippenhahn curve} of the matrix $A$. 

In algebraic-geometrical terms, $C(A)$ is the dual of a curve of degree $n$, which makes it a curve of {\em class} $n$. Class two curves are the same as curves of degree two, thus implying that for $A\in\M_2$ the Kippenhahn curve is an ellipse, in the case of normal $A$ degenerating into the doubleton of its foci.

A matrix $A\in\M_n$ is {\em tridiagonal}  if $a_{ij}=0$ whenever $\abs{i-j}>1$. We will be making use of the well known (and easy to prove) recursive relation for the determinants $\Delta_n$ of such matrices,
\eq{rec} \Delta_n=a_{nn}\Delta_{n-1}-a_{n-1,n}a_{n,n-1}\Delta_{n-2}. \en 

In this paper, we are interested mostly in tridiagonal matrices with an additional property that their main diagonal is constant: $a_{jj}:=a$, $j=1,\ldots,n$. Let us standardize the notation as follows:

\eq{tro} A(n;a,b,c)=\begin{bmatrix}a & b_1 & 0 & \ldots & 0 \\
	c_1 & a & b_2 & \ldots & 0 \\ \vdots & \ddots & \ddots & \ddots & \vdots \\ 0 & 0 & c_{n-2} & a & b_{n-1} \\ 0 & \ldots & 0 & c_{n-1} & a\end{bmatrix}. \en  
Since $W(A-\lambda I)=W(A)-\lambda$ for any $A\in\M_n$ and $\lambda\in\C$, it suffices to concentrate our attention on matrices $A(n;0,b,c)$ with the zero main diagonal. For $n>3$, we will restrict our attention further to what we will call {\em reciprocal} matrices. Namely, in addition to $A$ being of the form $A(n;0,b,c)$ we will also suppose that $b_jc_j=1$ for all $j=1,\ldots,n-1$, writing this symbolically as \eq{mrec} A=A(n;0,b,b^{-1}).\en 

Section~\ref{s:au} contains some properties of the NR generating polynomials and curves of tridiagonal, in particular reciprocal, matrices. A necessary condition for ellipticity of the numerical range is also established in this section, while its concrete implementations (along with the proof of sufficiency) for 4-by-4 and 5-by-5 reciprocal matrices are provided in Section~\ref{s:45}. The last two sections are devoted to 6-by-6 reciprocal matrices. In Section~\ref{s:big} criteria for the Kippenhahn curve of such matrices to contain an elliptical component are derived. By way of example it is also shown there that this component may be the exterior one thus guaranteeing the ellipticity of $W(A)$ without all the components of $C(A)$ being ellipses. Section~\ref{s:trel} in turn is devoted to the case when $C(A)$ consists of three ellipses and contains an example of a non-Toeplitz reciprocal matrix with this property. 

\section{Preliminary results}\label{s:au}

We start with some basic observations concerning NR generating polynomials of matrices \eqref{tro}.
\begin{prop}\label{th:sym}The NR generating polynomial of $A(n;0,b,c)$ is an even/odd function of $\lambda$ if $n$ is even (resp., odd). \end{prop} 
\begin{proof} Matrices $\re(e^{i\theta}A)-\lambda I$ are tridiagonal along with $A=A(n;0,b,c)$. More specifically, \[ \re(e^{i\theta}A(n;0,b,c))-\lambda I= A\left(n;-\lambda, (e^{i\theta}b+e^{-i\theta}\overline{c})/2, (e^{-i\theta}\overline{b}+e^{-i\theta}c)/2\right)\] and so \eqref{rec} implies  \begin{multline*} \det\left( \re(e^{i\theta}A(n;0,b,c))-\lambda I\right)
=-\lambda\det\left( \re(e^{i\theta}A(n-1;0,b,c))-\lambda I\right)\\ -\frac{\abs{e^{i\theta}b_{n-1}+e^{-i\theta}\overline{c_{n-1}}}^2}{4}\det\left( \re(e^{i\theta}A(n-2;0,b,c))-\lambda I\right).  \end{multline*} Since $\det\left( \re(e^{i\theta}A(1;0,b,c)\right)=-\lambda$ and \[ \det\left( \re(e^{i\theta}A(2;0,b,c)\right)=\lambda^2-{\abs{e^{i\theta}b_{1}+e^{-i\theta}\overline{c_{1}}}^2}/{4},\] the result follows by induction. 
\end{proof} 
\begin{cor}\label{co:sym} The Kippenhahn curve of $A=A(n;0,b,c)$ is central symmetric, for $n$ odd having the origin as one of its components.  \end{cor}
In particular, the numerical range of $A(n;0,b,c)$ is central symmetric --- a fact observed for the first time (to the best of our knowledge) in \cite[Theorem~1]{Chien96}. If $n=3$, Corollary~\ref{co:sym} implies that $C(A)$ consists of an ellipse centered at the origin and the origin itself. Thus, $W(A)$ is an elliptical disk (degenerating into a line segment if $A$ is normal) centered at the origin. This fact also follows from \cite[Theorem 4.2]{BS04}.

If $A$ is a reciprocal matrix, then   \eq{bc} \abs{e^{i\theta}b_{j}+e^{-i\theta}\overline{c_{j}}}^2= \abs{b_j}^2+\abs{b_j}^{-2}+2\cos(2\theta)=2(A_j+\tau), \quad j=1,\ldots,n-1,\en
where for notational convenience we relabeled 
\eq{Aj} \abs{b_j}^2+\abs{c_j}^2:=2A_j \text{ and } \cos(2\theta)=\tau. \en  
Note that $A_j\geq 1$, and the extremal case \eq{a1} A_1=\ldots = A_{n-1}=1 \en is easy, due to the following: 
\begin{prop}\label{th:norm}An $n$-by-$n$ reciprocal matrix $A$ is normal if and only if \eqref{a1} holds. If this is the case, then $A$ is in fact hermitian, and $W(A)$ is the real line segment with the endpoints $\pm 2\cos\frac{\pi}{n+1}$. \end{prop} 
\begin{proof} If a tridiagonal matrix \eqref{tro} is normal, the equalities $\abs{b_j}=\abs{c_j}$ can be obtained via a direct verification, or by applying the normality criterion from \cite[Lemma~5.1]{BS04}. For a reciprocal $A$ from $\abs{b_j}=\abs{c_j}$ it then follows that $\overline{b_j}=c_j$, thus making $A$ hermitian. Moreover, such $A$ is unitarily similar to a tridiagonal Toeplitz matrix \eqref{tro} with $a=0$ and $b_j=c_j=1$. The description of $W(A)$ then follows from the well known formula \[ \lambda_j=2\cos\frac{\pi j}{n+1}, \quad j=1,\ldots, n \] for the eigenvalues of $A$ (see e.g. \cite[Theorem~2.4]{BGru}). \end{proof}
\noindent In this case $C(A)=\sigma(A)$. 

On the other hand, the case \eq{a2} A_1=\ldots = A_{n-1}:=A_0>1 \en
is covered by \cite{BS04}. According to \cite[Theorem~3.3]{BS04}, $W(A)$ is then an elliptical disk. Following the proof of this theorem more closely, as in \cite[Section~5]{CDRSS}, reveals that in fact the ``hidden'' components of $C(A)$ are also all elliptical. 
\begin{thm}\label{th:toel}Let \eqref{a2} hold. Then $C(A)=\cup_{j=1}^{\lceil{n/2}\rceil}\sigma_j E$, where $E$ is the ellipse with the foci
$\pm 1$ and the axes having lengths $\sqrt{2(A_0\pm 1)}$, while 
\eq{sij} \sigma_j=\cos\frac{j\pi}{n+1}, \quad j=1,\ldots, \lceil{n/2}\rceil .  \en 
\end{thm} 
Note that the number of non-degenerate ellipses constituting $C(A)$ is $\lfloor n/2\rfloor$. Indeed, if $n$ is odd, then $\sigma_{\lceil n/2\rceil}=0$, and the respective component degenerates into the point $\{0\}$. 
 
The only difference in the statement of Theorem~\ref{th:toel} from the material already contained in \cite{CDRSS} is the explicit formula \eqref{sij} for the $s$-numbers $\sigma_j$ of the $\lceil{n/2}\rceil\times\lfloor{n/2}\rfloor$ matrix $X$ with the entries \[ x_{ij}=\begin{cases} 1 & \text{ if } i-j= 0, 1\\ 0 & \text{ otherwise.} \end{cases} \]

This covers in particular tridiagonal Toeplitz matrices, for which the ellipticity of their numerical ranges was established much earlier in \cite{E93}.   

In order to move beyond cases \eqref{a1} and \eqref{a2}, we need to further analyze properties of NR generating polynomials \eqref{nrgp}. 
Tracking the coefficients of $P_A(\lambda,\theta)$ in the proof of Proposition~\ref{th:sym} when matrices under consideration are reciprocal yields the following, more precise, statement. For convenience of notation, in addition to \eqref{Aj} we also introduce \[ \zeta=\lambda^2,\quad  k=\lfloor n/2\rfloor. \]

\begin{prop}\label{th:nrpr} Let $A\in\M_n$ be a reciprocal matrix. Then its NR generating polynomial has the form 
\eq{pn} P_n(\zeta,\tau)=\zeta^k+\sum_{j=0}^{k-1}p_j(\tau)\zeta^j, \en
premultiplied by $-\lambda$ in case $n$ is odd. Here $p_j$ are polynomials in $\tau$ of degree $k-j$ and coefficients depending only on $A_1,\ldots A_{n-1}$.
\end{prop} 
In particular, NR generating polynomials of reciprocal matrices are invariant under the change $\theta\mapsto -\theta$. Combined with the central symmetry of $C(A)$,  this observation implies  

\begin{cor}\label{th:sym1}Let $A$ be a reciprocal matrix. Then its Kippenhahn curve $C(A)$, and thus also the numerical range $W(A)$, is symmetric about both coordinate axes. \end{cor} 

We also need the following simple technical observation, which is a reformulation of a well known fact repeatedly used in the numerical range related literature (see, e.g., \cite{ChienHung} or \cite{Gau06}). It holds for all $A\in\M_n$, not just tridiagonal matrices. 

\begin{prop} \label{th:ell} The Kippenhahn curve $C(A)$ of $A\in\M_n$ contains an ellipse centered at the origin if and only if the NR generating polynomial of $A$ is divisible by \eq{ell} \lambda^2-(x\cos(2\theta)+y\sin(2\theta)+z)\en  with $x,y,z\in\R$ such that $z>\sqrt{x^2+y^2}$.  \end{prop} 

Combining the results stated above, we arrive at a necessary condition for a reciprocal matrix to have an elliptical numerical range. 

\begin{thm} \label{th:nec} Let $A$ be a reciprocal matrix with its numerical range being an elliptical disk. Then polynomial \eqref{pn} is divisible by $\zeta-(x\tau+z)$, where $\sqrt{z\pm x}$ are the lengths of the half-axes of $W(A)$. Equivalently, \eq{nescond} P_n(x\tau+z,\tau)=0 \text{ for all } \tau.\en \end{thm} 
\begin{proof} According to Corollary~\ref{th:sym1},  $W(A)$ has to be axes-aligned and centered at the origin. Since its boundary $E$ is a component of $C(A)$, Proposition~\ref{th:ell} implies that $P_n$ is divisible by \eqref{ell} with some triple $x,y,z$.  The geometry of $E$ means that $y=0$ and the values of $z\pm x$ are as described in the statement. Finally, \eqref{nescond} follows from the remainder theorem. \end{proof} 
\iffalse For our purposes, only ellipses with additional symmetry about the coordinate axes are of interest. This corresponds to $y=0$ in \eqref{ell}. \fi 

Treating $P_n(x\tau+z,\tau)$ as a polynomial in $\tau$, \eqref{nescond} can be rewritten as the system of $k+1$ algebraic equations in $x,z,A_1,\ldots,A_{n-1}$. So, the $(n-1)$-tuples $\{A_1,\ldots,A_{n-1}\}$ for which this system is consistent form an algebraic variety defined by $k-1$ polynomial equations. 
	 
\section{Reciprocal 4-by-4 and 5-by-5 matrices} \label{s:45} 
In full agreement with Proposition~\ref{th:nrpr}, polynomials \eqref{pn} for $n=4$ and $n=5$ are as follows: 

\eq{p4}
P_4(\zeta,\tau)=\zeta^2 -\frac{1}{2}\zeta \left(A_1+A_2+A_3+3\tau\right)\\ +
\frac{1}{4}(A_1+\tau ) (A_3+\tau )
\en
and 
\begin{multline}\label{p5}
P_5(\zeta,\tau)=
 \zeta^2-\zeta \left(\frac{1}{2} (A_1+A_2+A_3+A_4)+2\tau \right)  \\
+\frac{1}{4}\big((A_1+\tau ) (A_3+\tau ) 
+(A_1+\tau )(A_4+\tau ) +(A_2+\tau )(A_4+\tau )\big).
\end{multline}
Using the explicit formulas \eqref{p4},\eqref{p5}, in the next two theorems we restate the necessary condition of Theorem~\ref{th:nec} in a constructive way and show that it is also sufficient. The proofs of these results have a similar outline, varying in computational details only.

\begin{thm}\label{th:4by4}A reciprocal $4$-by-$4$ matrix $A$ has an elliptical numerical range if and only if 
\eq{4by4} A_2=\phi A_1-\phi^{-1}A_3 \text{ or }  A_2=\phi A_3-\phi^{-1}A_1,  \en 
	where $\phi=\frac{\sqrt{5}+1}{2}$ is the golden ratio,
and at  least one of the inequalities $A_j\geq 1$ ($j=1,2,3$) is strict. Moreover, in this case $C(A)$ is the union of two nested axis-aligned ellipses centered at the origin.  \end{thm} 

\begin{proof} {\sl Necessity.} Case \eqref{a1} is excluded due to Proposition~\ref{th:norm}. 
When divided by $\zeta-(x\tau+z)$, polynomial \eqref{p4} yields the quotient $\zeta-(x_1\tau+z_1)$, where
\[ x_1=\frac{3}{2}-x, \quad z_1=\frac{1}{2}(A_1+A_2+A_3)-z,  \] while condition \eqref{nescond} takes the 
form of \iffalse Furthermore, if $W(A)$ is an elliptical disk, then its boundary $\partial W(A)$ is a part of $C(A)$. From the symmetry of $W(A)$ and Proposition~\ref{th:ell} we conclude that then \eqref{p4} is divisible by $\lambda^2-(x\tau+z)$ for some choice of parameters $x,z$. The respective quotient is $\lambda^2-(x_1\tau+z_1)$, where
	\[ x_1=\frac{3}{2}-x, \quad z_1=\frac{1}{2}(A_1+A_2+A_3)-z. \] The remainder of this division being equal to zero amounts to\fi  the system:  
\begin{align*}
\begin{cases}
x\left(x-\frac{3}{2}\right)+\frac{1}{4}=0 \\
z\left(z-\frac{1}{2}(A_1+A_2+A_3)\right)+\frac{1}{4}A_1A_3=0 \\ 
(x-\frac{3}{2})z+x\left(z-\frac{1}{2}(A_1+A_2+A_3)\right)+\frac{1}{4}(A_1+A_3)=0.
\end{cases}
\end{align*}
Solving the first two equations for $x$ and $z$, respectively: 
\eq{xz} x=\frac{3\pm\sqrt{5}}{4},\ z=\frac{1}{4}(A_1+A_2+A_3\pm\sqrt{(A_1+A_2+A_3)^2-4A_1A_3}). \en
Plugging these values of $x$ and $z$ into the last equation of the system reveals that it is consistent if and only if the signs in \eqref{xz} match, and 
\eq{cons} \sqrt{(A_1+A_2+A_3)^2-4A_1A_3}=\frac{\sqrt{5}}{5}(A_1+3A_2+A_3). \en Finally, solving \eqref{cons} for $A_2$ yields \eqref{4by4}. 

{\sl Sufficiency.} Given \eqref{4by4} and choosing upper signs in \eqref{xz}, we arrive at the factorization of \eqref{p4} as $\left( \zeta-(x\tau+z)\right)\left(\zeta-(x_1\tau+z_1)\right)$. Condition $A_1+A_2+A_3>3$ guarantees that $(0<)\ x<z$. From here and the relation $4(x\tau-z)(x_1\tau-z_1)=(A_1+\tau)(A_3+\tau)$ we conclude that $C(A)$ consists of a non-degenerate ellipse $E$ and an ellipse $E_1$, degenerating into the doubleton of its foci if $A_1=1$ or $A_3=1$. Furthermore, $E_1$ lies completely inside $E$. Indeed,   \[ z-z_1=2z-\frac{1}{2}(A_1+A_2+A_3)=\frac{1}{2} \sqrt{(A_1+A_2+A_3)^2-4A_1A_3}\]
is bigger than $\frac{\sqrt{5}}{2}$ due to \eqref{cons} while  $x-x_1=2x-\frac{3}{2}=\frac{\sqrt{5}}{2}$, implying that  $x\tau+z> x_1\tau+z_1$ for all $\tau\in [-1,1]$. 

So, $W(A)$ is the elliptical disk bounded by $E$, both $E$ and $E_1$ are centered at the origin due to the central symmetry of $C(A)$, and their major axes are horizontal because $x,x_1>0$.  \end{proof}

Visualizing reciprocal 4-by-4 matrices as points $\{A_1,A_2,A_3\}$ in $\R_+^3$ we see that those with elliptical numerical ranges form a 2-dimensional manifold $\mathfrak M_4$ described by \eqref{4by4}. These equations show that $\mathfrak M_4$ contains the ray \eqref{a2}, as it should according to Theorem~\ref{th:toel}. The same equations imply the following
\begin{cor}\label{th:obs4} The intersection of $\mathfrak M_4$ with any of the bisector planes $A_i=A_j$ ($i\neq j$) consists only of the ray \eqref{a2}. \end{cor} 
It is worth mentioning that the ray $A_2=1, A_1=A_3>1$ corresponds to matrices with so called {\em bi-elliptical} numerical ranges, i.e. convex hulls of two non-concentric ellipses (see \cite[Theorem~7]{GeS2}).
 
\begin{thm}\label{th:5by5} A reciprocal $5$-by-$5$ matrix $A$ has an elliptical numerical range if and only if \eq{5by5} A_1=A_4 \text{ or } A_1-A_4=2(A_3-A_2) \en
and at  least one of the inequalities $A_j\geq 1$ ($j=1,\ldots,4$) is strict. Moreover, in this case $C(A)$ is the union of the origin and two nested axis-aligned ellipses centered there. 
\end{thm}
So, in contrast with the case $n=4$, a reciprocal matrix $A\in\M_5$ can have an elliptical numerical range while some but not all of its parameters $A_j$ coincide.
\begin{proof} 
The factor $-\lambda$ of the NR generating polynomial of $A$ corresponds to $\{0\}$ as a component of $C(A)$. Having duly noted that, we proceed by using \eqref{p5} in place of \eqref{p4} but otherwise basically along the same lines as in the proof of Theorem~\ref{th:4by4}, with the computations surprisingly being even simpler. 

{\sl Necessity.} Case \eqref{a1} is excluded due to Proposition~\ref{th:norm}. 
When divided by $\zeta-(x\tau+z)$, polynomial \eqref{p5} yields the quotient $\zeta-(x_1\tau+z_1)$, where \iffalse Furthermore, if $W(A)$ is an elliptical disk, then its boundary $\partial W(A)$ is a part of $C(A)$. From the symmetry of $W(A)$ and Proposition~\ref{th:ell} we conclude that then \begin{multline}\label{p51}
-\lambda^{-1}P_5(\lambda,\theta)=
\lambda^4-\lambda^2 \left(\frac{1}{2} (A_1+A_2+A_3+A_4)+2\tau \right)  \\
+\frac{1}{4}\big((A_1+\tau ) (A_3+\tau ) 
+(A_1+\tau )(A_4+\tau ) +(A_2+\tau )(A_4+\tau )\big)
\end{multline} is divisible by $\lambda^2-(x\tau+z)$ for some choice of parameters $x,z$. The respective quotient is $\lambda^2-(x_1\tau+z_1)$, where \fi 
\[ x_1=2-x, \quad z_1=\frac{1}{2}(A_1+A_2+A_3+A_4)-z, \] while condition \eqref{nescond} takes form of the system % and the remainder of this division being equal to zero amounts to the system:  
\begin{align*}
\begin{cases}
x(x-2)+\frac{3}{4}=0  \\
z\left(z-\frac{1}{2}(A_1+A_2+A_3+A_4)\right)+\frac{1}{4}(A_1A_3+A_1A_4+A_2A_4)=0 \\ 
(x-2)z+x\left(z-\frac{1}{2}(A_1+A_2+A_3+A_4)\right)+\frac{1}{4}(A_1A_3+A_1A_4+A_2A_4)=0.
\end{cases}
\end{align*}
Solving the first two equations for $x$ and $z$, respectively: 
\eq{xz1} x= 1\pm\frac{1}{2},\quad  z=\frac{1}{4}(A_1+A_2+A_3+A_4\pm D),\en
where \[ D=\sqrt{(A_1+A_2+A_3+A_4)^2-4(A_1A_3+A_1A_4+A_2A_4)}. \]
Plugging these values of $x$ and $z$ into the last equation of the system reveals that it is consistent if and only if the signs in \eqref{xz1} match, and $D=A_2+A_3$. The latter condition is equivalent to \eqref{5by5}. 

{\sl Sufficiency.} Given \eqref{5by5} and choosing upper signs in \eqref{xz1}, we arrive at the factorization of \eqref{p5} as \eq{p5f} P_5(\zeta,\tau)=\left( \zeta-(x\tau+z)\right)\left(\zeta-(x_1\tau+z_1)\right).\en Conditions $D=A_2+A_3\geq 2$ and  $A_1+A_2+A_3+A_4>4$ guarantee that $(0<)\ x<z$. From here and the relation \[  4(x\tau-z)(x_1\tau-z_1)=(A_1+\tau)(A_3+\tau)+(A_1+\tau)(A_4+\tau)+(A_2+\tau)(A_3+\tau) \] we conclude that $C(A)$ consists of $\{0\}$, a non-degenerate ellipse $E$ and a (possibly degenerating into a doubleton) ellipse $E_1$. Furthermore, $E_1$ lies inside $E$, with a non-empty intersection occurring only if $A_2=A_3=1$. Indeed,   \[ z-z_1=2z-\frac{1}{2}(A_1+A_2+A_3+A_4)=\frac{1}{2}D=\frac{1}{2}(A_2+A_3)\geq 1\]  while  $x-x_1=2x-2=1$. 
\end{proof} 
For convenience of future use, let us provide the explicit form of factorization \eqref{p5f} when \eqref{5by5} holds:
\eq{fac5}\hspace{0cm} P_5(\zeta,\tau)=\begin{cases} \left(\zeta-\frac{A_1+A_2+A_3+3\tau}{2}\right)\left(\zeta-\frac{A_1+\tau}{2}\right)
\\
 \left(\zeta-\frac{2A_3+A_4+3\tau}{2}\right)\left(\zeta-\frac{A_3+A_4-A_2+\tau}{2}\right)\end{cases}\en if $A_1-A_4$ equals zero or $2(A_3-A_2)$, respectively. 
\begin{Ex}
For the $5\times 5$ reciprocal matrix \eqref{mrec} with the superdiagonal string $b=(1.5, 2 ,2.5, 1.5)$ we have
\[ A_1=A_4=1.347\overline{2},\ A_2=2.125,\ A_3=3.205, \]
and according to the upper line of \eqref{fac5} \[ P_5(\zeta,\tau)=\left(\zeta-\frac{ 6.67722+3\tau}{2} \right)\left(\zeta-\frac{1.34722+\tau}{2}\right).
\]
The Kippenhahn curve of this matrix, consisting of two ellipses and the origin, is shown in Figure~\ref{fig95E1}. 
\begin{figure}[H]
		\centering
		\includegraphics[width=0.75\linewidth,angle=0]{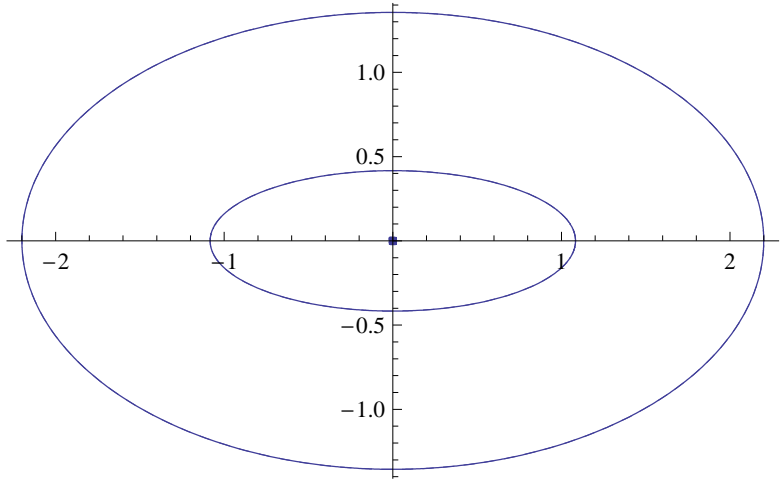}
		%{ellipticity.eps}
		\caption{$b_1=1.5,~b_2=2,~b_3=2.5,~b_4=1.5.$ The curves are exactly elliptical.}
		\label{fig95E1}
	\end{figure}

Now let $A$ be a reciprocal matrix \eqref{mrec} with $b=(1.5,2,2,3)$. Then $A_1-A_4\approx -3.2$, $A_3-A_2=0$, and so \eqref{5by5} does not hold.  The Kippenhahn curve of this matrix is shown in Figure~\ref{fig95NE}. 
Its components are non-elliptical; the best fitting ellipses are pictured as  dotted curves.
\end{Ex} 
	\begin{figure}[h]
		\centering
		\includegraphics[width=0.75\linewidth,angle=0]{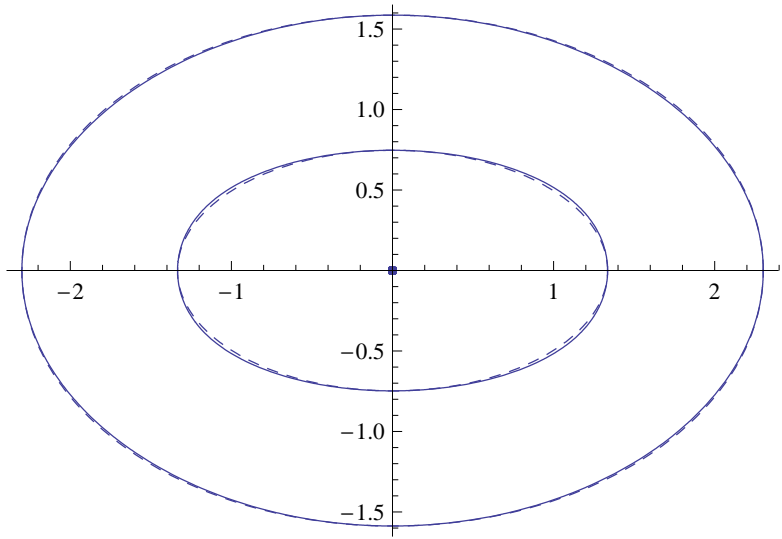}
		%{ellipticity.eps}
		\caption{$b_1=1.5,~b_2=2,~b_3=2,~b_4=3$.  The curves look elliptical but they are not  exactly elliptical.\iffalse The deviation from ellipticity is  measured as 0.0120462.\fi}
		\label{fig95NE}
	\end{figure}

\section{Reciprocal 6-by-6 matrices with $C(A)$ containing an ellipse} \label{s:big}
When $n$ increases, things get more complicated. It may become impossible to state the divisibility condition from Proposition~\ref{th:ell} in exact arithmetic. Besides, this condition does not any longer guarantee the ellipticity of $W(A)$. We will illustrate these phenomena for $n=6$. 

Formula \eqref{pn} then takes the form:
\begin{multline} \label{p6} 
P_6(\zeta,\tau)= \zeta^3-\frac{1}{2}\zeta^2 \left(A_1+A_2+A_3+A_4+A_5+5\tau\right)\\ +\frac{1}{4}\zeta
\left(A_1\left(A_3+A_4+A_5\right)+A_2 \left(A_4+A_5\right)+A_3 A_5+\left(3(A_1+A_5)\right.\right. \\ +2\left.\left.(A_2+A_3+A_4)\right)\tau+6\tau^2\right) 
-\frac{1}{8} \left(A_1+\tau \right) \left(A_3+\tau \right) \left(A_5+\tau \right).
\end{multline} 
Consequently, 
\[  P_6(x\tau+z,\tau)=\frac{1}{8}\tau^3\left(8x^3-20x^2+12x-1\right)+\tau^2 Q_1(x,z)+\tau Q_2(x,z)+Q_3(x,z).\] 
Here 
\eq{z} Q_1=q_{11}z+q_{10}, \ Q_2=q_{22}z^2+q_{21}z+q_{20},\  Q_3= z^3+q_{32}z^2+q_{31}z+q_{30} \en 
with $q_{ij}$ given by
\begin{align*}
q_{11} = & 3 x^2-5 x+\frac{3}{2},\quad  q_{22} =  3 x-\frac{5}{2}, \\
q_{10} = & -\frac{1}{2}\left(A_1+A_2+A_3+A_4+A_5\right)
x^2\\ & +\frac{1}{4}\left(3 A_1+2A_2+2A_3+2A_4+3 A_5\right)x -\frac{1}{8}\left(A_1+A_3+A_5\right), \\ 
q_{21} = & -\left(A_1+A_2+A_3+A_4+A_5\right) x+\frac{1}{4}
	\left(3A_1+2A_2+2A_3+2A_4+3A_5\right), \\
q_{20} = &\frac{1}{4}\left(A_1 A_3+A_5 A_3+A_1 A_4+A_2 A_4+A_1A_5+A_2 A_5\right) x\\ & -\frac{1}{8}\left(A_1 A_3+A_5 A_3+A_1 A_5\right), \\
q_{32} = & -\frac{1}{2}\left(A_1+A_2+A_3+A_4+A_5\right),\\
%\end{align*}\begin{align*} 
q_{31} = & \frac{1}{4}\left(A_1 A_3+A_5 A_3+A_1 A_4+A_2 A_4+A_1A_5+A_2 A_5\right), \\
q_{30} = & -\frac{1}{8} A_1 A_3 A_5.
\end{align*}
Condition \eqref{nescond} therefore is nothing but the requirement for all three $Q_j$ from \eqref{z} to vanish at the same root of the cubic equation 
\eq{x} 8x^3-20x^2+12x-1= 0. \en 
In turn, this happens if and only if the resultant $R_{1}$ of $Q_1$ and $Q_2$, as well as the resultant  $R_{2}$ of $Q_1$ and $Q_3$, are equal to zero. A somewhat lengthy but straightforward computation, with the repeated use of \eqref{x}, reveals that $R_{1}, R_{2}$ can be treated as quadratic polynomials in $x$. Namely,
\eq{res} R_{j}(x)=r_{j2}x^2+r_{j1}x+r_{j0},\quad j=1,2, \en 		
where 	
\begin{align*}
	r_{12} = & -8(A_1^2 + A_2A_3+A_3A_4+A_5^2)-28(A_1A_2-A_1A_4-A_2A_5+A_4A_5)\\&
	+4(A_1A_3+A_3A_5+A_3^2)-12(A_2^2+A_4^2-A_1A_5)+32A_2A_4, \\
	r_{11} = & 12(A_2^2+A_4^2)+18(A_1A_2-A_2A_4+A_3A_5)-24(A_1A_4+A_2A_5)\\&
	+6(A_1A_3-A_3^2+A_3A_5),\\
	r_{10} = & 4(A_1^2-A_2^2-A_4^2+A_5^2+A_2A_3+A_3A_4)-(A_1A_2+A_2A_4+A_4A_5)\\&
	-7(A_1A_3+A_3A_5)+6(A_1A_4-A_1A_5+A_2A_5)+3A_3^2
\end{align*}
and
\begin{align*}
r_{22} = & -12(A_1^3+A_5^3) -48(A_1^2A_2+A_4A_5^2) -60(A_1A_2^2+A_4^2A_5)
	-16(A_2^3+A_4^3) \\& +44(A_1A_3^2-A_1A_2A_3+A_1A_2A_5 +A_3^2A_5+A_1A_4A_5-A_3A_4A_5) \\&
	-36(A_2^2A_3-A_2^2A_4-A_2A_4^2+A_3A_4^2) -24(A_2A_3^2+A_3^2A_4-A_1A_4^2-A_2^2A_5)\\&
	+8(A_1^2A_4+A_2A_5^2) +20(-A_1A_2A_4+A_1^2A_5+A_2A_4A_5+A_1A_5^2)\\&
	+68(A_1A_3A_4+A_2A_3A_5)-4A_3^3+40A_2A_3A_4-84(A_1A_3A_5),\\
r_{21}= & 14(A_1^3A_5^3-A_1A_4^2-A_2^2A_5) +56(A_1A_2^2+A_4^2A_5)
	+16(A_2^3+A_4^3)\\& -12(A_1^2A_3+A_3A_5^2) +30(A_2^2A_3+A_3A_4^2) -50(A_1A_3^2+A_3^2A_5)\\&
	+20(A_2A_3^2+A_3^2A_4)-22(A_1^2+A_2^2A_4+A_2A_4^2+A_2A_5^2) \\&
	-28(A_1A_2A_4+A_1^2A_5+A_2A_4A_5+A_1A_5^2) -66(A_1A_3A_4+A_2A_3A_5) \\&
	-44(A_1A_2A_5+A_1A_4A_5) +48(A_1^2A_2+A_4A_5^2) +46(A_1A_2A_3+A_3A_4A_5) \\&
	+6A_3^3+158A_1A_3A_5-52A_2A_3A_4,\\
\end{align*}\begin{align*} 
r_{20} = & -2(A_1^3+A_2^3+A_3^3+A_4^3+A_5^3+A_2A_3^2+A_3^2A_4)\\&
	-4(A_1^2A_2+A_1A_2^2+A_1A_4^2+A_2^2A_5+A_4^2A_5+A_4A_5^2)\\&
	+6(A_1^2A_3+A_1A_2A_4-A_2^2A_4-A_2A_4^2+A_1A_2A_5+A_1A_4A_5 +A_2A_4A_5+A_3A_5^2)\\&
	-10(A_1A_2A_3+A_3A_4A_5-A_1^2A_4-A_2A_5^2) -(A_2^2A_3+A_3A_4^2)\\&
	+12(A_1A_3^2+A_3^2A_5) +11(A_1A_3A_4+A_2A_3A_5) +8(A_1^2A_5+A_1A_5^2)\\&
	+19A_2A_3A_4-65A_1A_3A_5.
\end{align*}
We thus arrive at the following 
\begin{thm} \label{th:ell6} A reciprocal $6$-by-$6$ matrix $A$ has an elliptical component in its Kippenhahn curve $C(A)$ if and only if both polynomials \eqref{res} share a common root with equation \eqref{x}. 
\end{thm} 
Note that the requirement on the elliptical component of $C(A)$ to be centered at the origin is not included in the statement of Theorem~\ref{th:ell6}. It holds automatically for the following reason: if $C(A)$ contains an ellipse $E$ not centered at the origin, then due to the symmetry it also contains $-E$ which is different from $E$. But them the third component of $C(A)$ also has to be elliptical and, since there cannot be four, this one is then centered at the origin.

Plugging in the approximate values \[ x_1 \approx  0.0990311,\ x_2 \approx  0.777479,\ x_3 \approx  1.62349 \]
for the roots of \eqref{x} into \eqref{res} thus delivers the numerical tests for $C(A)$ to contain an elliptical component. Due to the structure of $r_{ij}$, these are systems of two homogeneous polynomial equations (one quadratic and one cubic) in five unknowns $A_1,\ldots,A_5$. Fixing any three of $A_j$ (or otherwise reducing the number of free parameters to two) and solving for the other two yields the specific examples of matrices satisfying Theorem~\ref{th:ell6}.
\begin{Ex}Under additional constraints $A_1=A_5:=uA_3$, $A_2=A_4:=vA_3$ the polynomials \eqref{res} both have the root $x=x_3$ if and only if the pairs $(u,v)$ are as follows: $(1,1)$, $(8.84369,-2.49077)$, $(-1.80414,2.24796)$, and \eq{uv}(1.7724359313231006,\ 0.6562336702811362).\en  

The first solution corresponds to the situation \eqref{a2} in which all three components of $C(A)$ are elliptical; the next two are irrelevant because of non-positivity. The remaining pair \eqref{uv} delivers a non-trivial example in which one component of $C(A)$ is an ellipse while two others are not (the latter fact follows from Theorem~\ref{th:A1245} below). The respective $C(A)$ is shown in Figure~\ref{fig.oneel}. Only the outer component is an exact ellipse, though deviations of two others from being elliptical are almost negligent.
\end{Ex} 
\begin{figure}[h]
	\centering
	\includegraphics[width=0.75\linewidth,angle=0]{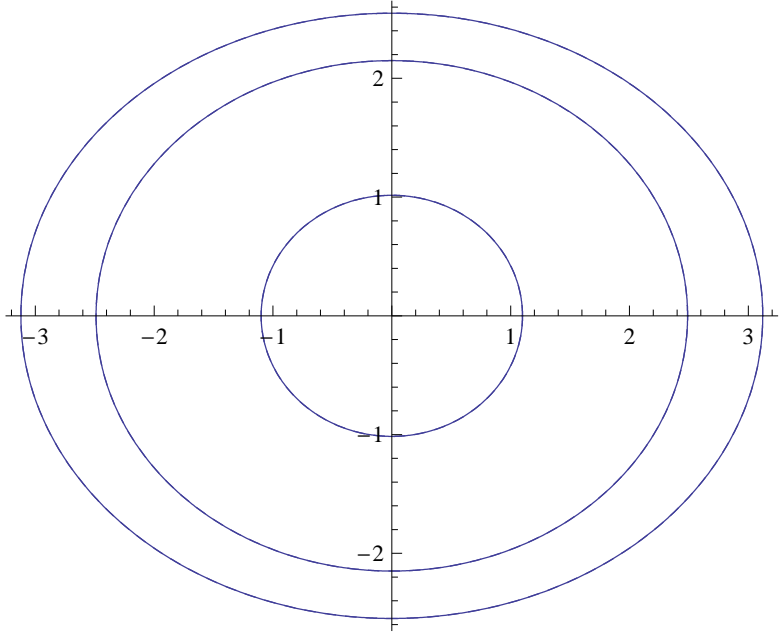}
	%{ellipticity.eps}
	\caption{$C(A)$ for $A_1=A_5=3.28117,\, A_2=A_4=6.5623367,\, A_3=5$.}
	\label{fig.oneel}
\end{figure}

\section{Reciprocal 6-by-6 matrices with $C(A)$ consisting of three ellipses} \label{s:trel} 	
From Theorem~\ref{th:ell6} it follows that a reciprocal matrix $A\in\M_6$ has $C(A)$ consisting of three concentric ellipses (and thus an elliptical numerical range) if and only if $R_1(x_j)=R_2(x_j)=0$ for all three roots of \eqref{x}. This means exactly that $r_{jk}=0$ for all $j=1,2;\ k=0,1,2$. 

We have arrived at the system of six polynomial equations in five variables $A_1,\ldots, A_5$, which therefore seems overdetermined. As was already observed in Section~\ref{s:big}, the ellipticity of two components of $C(A)$ automatically implies that the third one is an ellipse, centered at the origin. So, the number of equations can be easily reduced to four. An alternative approach shows, however, that in fact reciprocal 6-by-6 matrices with $C(A)$ consisting of three ellipses are characterized by just three equations, and so form a two-parameter set.  

To describe this approach, observe that \eqref{p6} factors as $\prod_{j=1}^3\left(\zeta-(x_j\tau+z_j)\right)$ if and only if 
\begin{align*}
\begin{cases}
z_1 + z_2 + z_3 = \frac{1}{2} (A_1+A_2+A_3+A_4+A_5)\\
z_1z_2z_3= \frac{1}{8} A_1A_3A_5\\
z_1z_2+z_1z_3+z_2z_3 = \frac{1}{4} (A_1A_3+A_1A_4+A_1A_5+A_2A_4+A_2A_5+A_3A_5)\\
x_1 + x_2 + x_3 = \frac{5}{2} \\
x_1x_2x_3 = \frac{1}{8}\\
x_1x_2+x_1x_3+x_2x_3 = \frac{3}{2}\\
z_1x_2x_3 + z_2x_1x_3 + z_3 x_1x_2 = \frac{1}{8}(A_1+A_3+A_5)\\
z_1z_2x_3 + z_1z_3x_2 + z_2z_3x_1 = \frac{1}{8} (A_1A_5+A_1A_3+A_3A_5)\\
z_1(x_2+x_3)+z_2(x_1+x_3)+z_3(x_1+x_2) = \frac{3}{4}(A_1+A_5)+\frac{1}{2}(A_2+A_3+A_4).
\end{cases}
\end{align*}
The three equations of this system not containing variables $z_j$ mean exactly that $x_1,x_2,x_3$ are the roots of \eqref{x}. The three equations linear in $z_j$ can be rewritten as

\[ \begin{bmatrix}
1 & 1 & 1\\
x_2 + x_3 & x_1 + x_3 & x_1+x_2\\
x_2x_3 & x_1x_3 & x_1x_2 
\end{bmatrix}
\begin{bmatrix}
z_1 \\
z_2\\
z_3
\end{bmatrix}
= \begin{bmatrix}
\frac{1}{2} (A_1+A_2+A_3+A_4+A_5)\\
\frac{3}{4}(A_1+A_5)+\frac{1}{2}(A_2+A_3+A_4)\\
\frac{1}{8}(A_1+A_3+A_5)
\end{bmatrix}.
\]

Solving this system:
\eq{zj} z_j=\frac{R(x_j)}{8(x_j-x_i)(x_j-x_k)},\quad j=1,2,3, \en 
where $\{i,k\}=\{1,2,3\}\setminus\{j\}$ and $R(x)=r_2x^2+r_1x+r_0$ with \begin{multline}\label{R} r_2=4(A_1+A_2+A_3+A_4+A_5),\\ r_1=4(A_2+A_3+A_4)-6(A_1+A_5),\quad r_0=A_1+A_3+A_5. \end{multline}
 
Finally, the remaining three nonlinear equations in $z_j$ yield the following:
\begin{align*}
& 16 r_0^2+24 \left(r_1+r_2\right) r_0+12 r_1^2+15 r_2^2+28 r_1
	r_2\\&
	+112\left(A_1 A_3+A_5 A_3+A_1 A_4+A_2 A_4+A_1A_5+A_2 A_5\right)=0;\\ & 
32 r_0^2+8 \left(3 r_1+r_2\right) r_0+4 r_1^2+3 r_2^2+6 r_1r_2+112\left(A_1 A_3+A_5 A_3+A_1 A_5\right)=0;
\\	&\frac{1}{64} \Big(64 r_0^3+16 (10 r_1+13 r_2) r_0^2+8 (12 r_1^2+27 r_2
r_1+13 r_2^2) r_0+8 r_1^3+r_2^3\\&
+12 r_1 r_2^2+20 r_1^2 r_2\Big)+49 A_1 A_3 A_5=0
\end{align*}
or, plugging in the values of $r_j$ from \eqref{R}: 
\eq{tel2}\begin{aligned}
&  -(A_2^2+A_4^2)-2(A_1^2 + A_5^2)+2A_3(A_1-A_2-A_4+A_5)\\& 
	+3(A_1A_4+A_1A_5+A_2A_5) -4(A_1A_2+A_4A_5) +5A_2A_4=0,
\end{aligned}\en 
\eq{tel3}\begin{aligned}
&  A_2^2-A_3^2+A_4^2+3(A_1A_3+A_1A_5+A_3A_5)\\& 
		-2(A_1^2+A_5^2+A_2A_3-A_2A_4+A_3A_4)-(A_1+A_5)(A_2+A_4) =0,
	\end{aligned}\en and 
\eq{tel1}
\begin{aligned}
 & -A_1^3+A_2^3+A_3^3+A_4^3-A_5^3
	+(A_1-A_3+A_5)(A_2^2+A_4^2)\\&
	-2(A_2+A_4)(A_1^2+A_3^2+A_5^2)
	+2A_2A_4(A_1-A_3+A_5)\\& -3(A_2+A_4)\left(A_3(A_1+A_5)-A_2A_4\right)
	-3(A_1+A_5)(A_3^2+A_1A_5)\\& -4(A_1^2A_3+A_3A_5^2+A_1A_2A_5+A_1A_4A_5)+41A_1A_3A_5=0, 
\end{aligned}\en
respectively. 

We have thus reached the conclusion.
\begin{thm}\label{th:3el}A reciprocal 6-by-6 matrix $A$ has its Kippenhahn curve consisting of three concentric ellipses if and only if  $A_j$ defined by \eqref{Aj} satisfy \eqref{tel2}--\eqref{tel1}, and in addition $A_1+\cdots+A_5>1$.\end{thm}
Note that the difference of \eqref{tel2} and \eqref{tel3}, 
\eq{tel4} \begin{aligned}
& A_3^2-A_3(A_1+A_5)-2(A_2^2+A_4^2)-3(A_1A_2+A_4A_5) \\& 
	+3A_2A_4+4(A_1A_4+A_2A_5)=0,
\end{aligned}\en 
is somewhat simpler than either of these conditions, and for computational purposes it thus might be useful to replace \eqref{tel2} or \eqref{tel3} (but not both) with \eqref{tel4}. 

The semialgebraic subset ${\mathfrak M}_6$ of $\R^5_+$ defined by the system \eqref{tel2}--\eqref{tel1} is of a more complicated structure than the pairs of hyperplanes \eqref{4by4},\eqref{5by5} corresponding to $n=4,5$ cases. Even though we do not have an explicit description of ${\mathfrak M}_6$, a direct verification confirms that it contains the ray
$A_1=\ldots=A_5>0$ --- as it should, in order to agree with Theorem~\ref{th:toel}. Our next result, somewhat analogous to Corollary~\ref{th:obs4}, shows that this ray is in fact the intersection of ${\mathfrak M}_6$ with either of the hyperplanes $A_2=A_4$ and $A_1=A_5$.
\begin{thm}\label{th:A1245}If $(A_1,\ldots,A_5)\in{\mathfrak M}_6$ and $A_2=A_4$ or $A_1=A_5$, then all $A_j$ coincide.\end{thm}
\begin{proof}Equations \eqref{tel2}--\eqref{tel4} are homogeneous, so by scaling without loss of generality we may set $A_3=1$. 
	
{\sl Case 1.}  $A_2=A_4:=s$. Relabeling also for convenience of notation $A_1=x, A_5=y$, we may rewrite \eqref{tel4} as  
\[ (s-1)(x+y-(s+1))=0, \] and conclude from there that $s=1$ or $x+y=s+1$.

In turn, \eqref{tel3} in our abbreviated notation amounts to 
\eq{tel31} 4s^2-4s-1+(3-2s)(x+y)-2(x^2+y^2)+3xy=0. \en
If $s=1$, \eqref{tel31} simplifies further to  \[ 2(x^2+y^2)-3xy-(x+y)+1=0, \]
which for $x,y>0$ is possible only when $x=y=1$. 

On the other hand, plugging $x+y=s+1$ into \eqref{tel31} yields $xy=s$, so that \eq{xys} x=s,\ y=1, \text{ or } x=1,\ y=s.\en  Since \eqref{tel1} in our setting  is nothing but 
\[ 8s^3-4s^2-4s+1-(x+y)^3-4s(x+y)^2+(4s^2-6s-3)(x+y)-4(x^2+y^2)+41xy=0, \]
in view of \eqref{xys} it amounts to $7(s-1)^3=0$, thus implying $s=1$. According to \eqref{xys}, then also $x=y=1$.

{\sl Case 2.} $A_1=A_5:=s$. Letting now $A_2=x, A_4=y$, rewrite \eqref{tel3} as 
\[ \left(x+y-(s+1)\right)^2-2(s-1)^2=0. \] So, \eq{x+y} x+y=s+1\pm\sqrt{2}(s-1). \en

On the other hand, in our abbreviated notation, \eqref{tel1} is nothing but 
\[ (x + y)^3 + (2 s - 1) (x + y)^2 - (8 s^2 + 6 s+2)(x + y)-8s^3+ 33 s^2-6s+1=0\] 
or, plugging in the expression for $x+y$ from \eqref{x+y}:
\begin{multline*} -1 - 10 s + 25 s^2 - 8 s^3 - (s \pm \sqrt{2} (-1 + s)) - 2 s (s \pm \sqrt{2} (-1 + s)) -  8 s^2 (s \pm \sqrt{2} (-1 + s))\\ + 2 (s \pm \sqrt{2} (-1 + s))^2 + 2 s (s \pm \sqrt{2} (-1 + s))^2 + (s \pm\sqrt{2} (-1 + s))^3=0. \end{multline*} 
The left hand side of the latter equation is simply $(3\pm\sqrt{2})(s-1)^3$, and therefore $s=1$. Equation\eqref{tel4} therefore takes the form
\[ -1-2(x+y)^2+x+y+7xy=0. \]
Observing that $x+y=2$ due to \eqref{x+y} and $s=1$, we conclude that $xy=1$, and so in fact $x=y=1$. 
\end{proof} 
Non-trivial elements of $\mathfrak M_6$ can be constructed as follows. Fix two of the parameters $A_1,A_2,A_4,A_5$ (making sure to avoid the $A_1=A_5$ or $A_2=A_4$ situations), and solve \eqref{tel2}--\eqref{tel3} for the other two, expressing the solutions as function of $A_3$. Then find $A_3$ as a solution to \eqref{tel1}.

\begin{Ex}
Set $A_1=20, A_5=40$. Then \[ A_2=64.9396,\ A_3=36.9387548,\ A_4=28.9008 \] deliver a solution to \eqref{tel2}--\eqref{tel1} The respective $C(A)$ curve is plotted in Figure~\ref{fig95outro}. 
\end{Ex}
\begin{figure}[H]
	\centering
	\includegraphics[width=0.75\linewidth,angle=0]{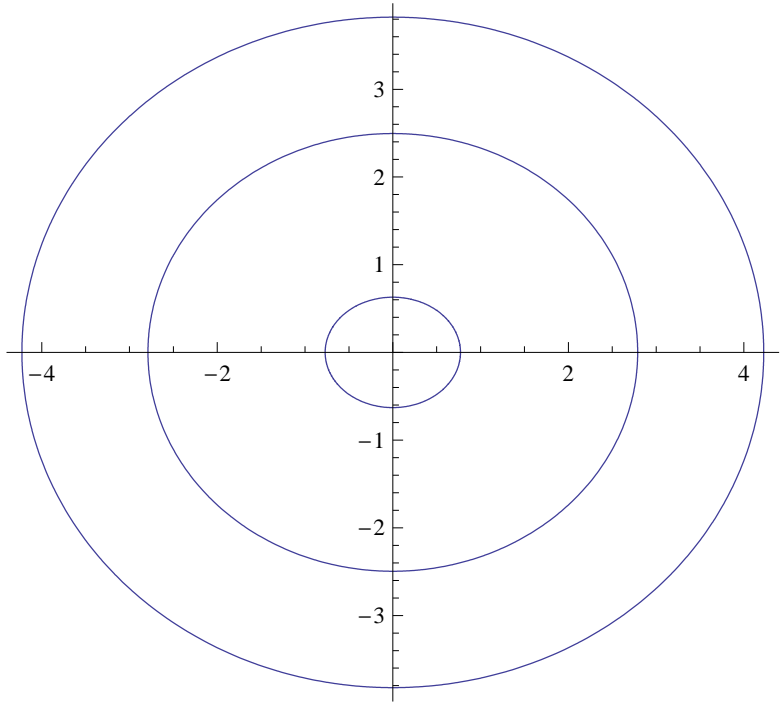}
	%{ellipticity.eps}
	\caption{$C(A)$ for $A_1= 10,~A_2= 32.4698,~A_3=18.4693774,~A_4= 14.4004,~A_5=20$.}
	\label{fig95outro}
\end{figure}

\providecommand{\bysame}{\leavevmode\hbox to3em{\hrulefill}\thinspace}
\providecommand{\MR}{\relax\ifhmode\unskip\space\fi MR }
% \MRhref is called by the amsart/book/proc definition of \MR.
\providecommand{\MRhref}[2]{%
	\href{http://www.ams.org/mathscinet-getitem?mr=#1}{#2}
}
\providecommand{\href}[2]{#2}

%\bibliographystyle{amsplain}
%\bibliography{../../Working_files/master}
\end{document}